\documentclass[twoside]{siamltex}

\usepackage{amssymb,amsmath,exscale}
\usepackage{xspace}
\usepackage[draft,ulem=normalem]{changes}

\newtheorem{remark}[theorem]{{\itshape Remark}}
\newcommand{\Pb}{\mbox{\rm (P)}\xspace}
\newcommand{\Pbe}{\mbox{\rm (P$_{\epsilon}$)}\xspace}

\newcommand{\mo}{{\mathcal{M}(\omega)}}
\newcommand{\mon}{{\mathcal{M}(\omega)^n}}
\newcommand{\BV}{{BV(\omega)}}
\newcommand{\C}{{C_0(\omega)}}
\newcommand{\Cn}{{C_0(\omega)^n}}
\renewcommand{\div}{\operatorname{\text{\rm div}\!}}

 \title{Analysis of Optimal Control Problems of Semilinear Elliptic Equations by BV-Functions
\thanks{The first author was supported by Spanish Ministerio de Econom\'{\i}a y Competitividad under project MTM2014-57531-P. The second was supported by the ERC advanced grant 668998 (OCLOC) under the EU’s H2020 research program.}}

\author{Eduardo Casas\thanks{Departamento de Matem\'{a}tica Aplicada y Ciencias de la Computaci\'{o}n, E.T.S.I. Industriales y de Telecomunicaci\'on, Universidad de Cantabria, 39005 Santander, Spain (eduardo.casas@unican.es).}
    \and Karl Kunisch\thanks{Institute for Mathematics and Scientific Computing, University of Graz, Heinrichstrasse 36, A-8010 Graz, Austria (karl.kunisch@uni-graz.at).}}

\begin{document}

%\tikzset{external/force remake}

\maketitle

\begin{abstract}
Optimal control problems for semilinear elliptic equations with control costs in the space of bounded variations are analysed. BV-based optimal controls favor  piecewise constant, and hence 'simple' controls,  with few  jumps.  Existence of optimal controls, necessary and sufficient optimality conditions of first and second order are analysed. Special attention is paid on the effect of the choice of the vector norm in the definition of the BV-seminorm for the optimal primal and adjoined variables.
\end{abstract}

\begin{AMS}
35J61, % semilinear elliptic equations
%35K58,  % semilinear parabolic equations
49J20, % optimal control problems involving partial differential equations
49J52, % nonsmooth analysis
49K20, % optimality conditions: problems involving partial differential equations
%49M25, discrete approximations
%90C48, % programming in abstract spaces
\end{AMS}

\begin{keywords}
optimal control, bounded variation functions, sparsity, first and second order optimality conditions, semilinear elliptic equations
\end{keywords}

\pagestyle{myheadings} \thispagestyle{plain} \markboth{E6.~CASAS AND K.~KUNISCH}{Optimal Control by BV-Functions}

%% main matter

\section{Introduction}
\label{S1}
This paper is dedicated to the study of the optimal control problem
\[
   \Pb \ \ \ \min_{u \in \BV} J(u) = \frac{1}{2}\int_\Omega|y - y_d|^2\, dx + \alpha\int_\omega|\nabla u| + \frac{\beta}{2}\Big(\int_\omega u(x)\, dx\Big)^2 + \frac{\gamma}{2}\int_\omega u^2(x)\, dx,
\]
where $y$ is the unique solution to the Dirichlet problem
\begin{equation}
\left\{\begin{array}{rccl}-\Delta y + f(x,y) & = & u\chi_\omega & \text{in } \Omega,\\y & = & 0 & \text{on } \Gamma.\end{array}\right.
\label{E1.1}
\end{equation}
The control domain $\omega$ is an open subset of $\Omega$. We assume that $\alpha > 0$, $\beta \ge 0$, $\gamma \ge 0$, $y_d \in L^2(\Omega)$, and $\Omega$ is a bounded domain in $\mathbb{R}^n$, $n = 2$ or $3$, with Lipschitz boundary $\Gamma$. Additionally we make the following hypothesis:
\begin{equation}
\text{if } n = 3, \ \text{ then } \ \gamma > 0 \ \text{ is assumed.}
\label{E1.2}
\end{equation}

Here, $\BV$ denotes the space of functions of bounded variation in $\omega$ and $\int_\omega|\nabla u|$ stands for the total variation of $u$. The assumptions on the nonlinear term $f(x,y)$ in the state equation will be formulated later. By introducing the penalty term involving the mean of $u$ when $\beta>0$ we realize the fact that constants functions constitute the kernel of the BV-seminorm. {If $\gamma = 0$,} in dependence on the order of the nonlinearity $f$ it can be necessary to choose $\beta>0$ to guarantee that \Pb admits a solution.

The use of the BV-seminorm in \Pb  enhances that the optimal controls are piecewise constant in space. Thus the cost functional in  \Pb models the objective of simultaneously determining   a control  of simple structure and  resulting  in a state $y=y(u)$ which is as close to  $y_d$ as possible. Comparing with the common formulation of using $L^2(\omega)$ or $L^p(\omega)$ control-cost functionals, with $p>2$ to match the nonlinearity $f$, these later functionals will produce smooth optimal controls which may be more intricate to realize in practice than controls which result from the $BV-$formulation. Piecewise constant behavior of the optimal  controls can also be obtained by introducing bilateral bounds $  \underline a\le   u(x) \le \bar b$ together with only the tracking term in \Pb.  In this case we can expect optimal controls which exhibit bang-bang structure. If an $L^1(\omega)$ control cost term is added then the optimal control will be of the form bang-zero-bang.  But this type of behavior is distinctly different from that which is allowed in \Pb, since the value of the piecewise constants plateaus is not prescribed. This is distinctly different from the bilaterally constraint case where the optimal control typically assumes one of the extreme values $\underline a$ or $\bar b$. This in turn can lead to unnecessarily high control costs.

Possibly one of the first papers where this was pointed out, but not systematically investigated is \cite{ClasonKunisch2010}. In \cite{CKK2016} semilinear parabolic equations with temporally dependent BV-functions as controls were investigated.  Thus we were focusing on controls which are  optimally switching in time. The analysis for this case is simpler and exploits specific properties of  BV-functions in dimension one. Numerically the simple structure of the controls which is obtained for BV-constrained control problems was already demonstrated in \cite{CCK12, CKK2016} and a recent master thesis \cite{DH17}.  BV-seminorm control costs are also employed in \cite{CasasKogutLeugering}, where the control appears as coefficient in the $p$-Laplace equation. Beyond these papers the choice of the control costs related to BV-norms or BV-seminorms has not received much attention in the optimal control literature yet.

In mathematical image analysis, to the contrary, the BV-seminorm has received a tremendous amount of attention. The beginning of this activity is frequently dated to \cite{ROF1991}. Let us also mention the recent paper \cite{BrediesHoller} which gives interesting insight into the structure of the subdifferential of the BV-seminorm. Fine properties of BV-functions, in the context of image reconstruction problems, in particular the  stair casing effect were, analyzed for the one-dimensional case in  \cite{Ring2000}, and in higher dimensions in \cite{J16,CDP17}, for example. In \cite{CKP98} the authors provided a convergence analysis for BV-regularized mathematical imaging problems by finite elements, paying special attention to the choice of the vector norm in the definition of the BV-seminorm.

Let us also compare the use of the BV-term in \Pb with the efforts that have been made for studying optimal control problems with sparsity constraints. These formulations involve either measure-valued norms of the control or $L^1$-functionals combined with pointwise constraints on the control. We cite \cite{CCK12,CHW12}  from among the many results which are now already available.  The BV-seminorm therefore can also be understood as a sparsity constraint for the first derivative.

Let us briefly describe the structure of the paper. Section 2 contains an analysis of the state equation and the smooth part of the cost-functional.  The non-smooth part of the cost-functional is investigated in Section 3. Special attention is given to the consequences which arise from the specific choice which is made for the vector norm in the variational definition of the BV-seminorm. In particular, we consider the Euclidean and the infinity norms. Existence of optimal solutions and first order optimality conditions are obtained in Section 4. Second order sufficient optimality conditions are provided in Section 5. Finally in Section 6 we consider \Pb with an additional $H^1(\omega)$ regularisation term and investigate the asymptotic behavior as the weight of the $H^1(\omega)$ regularisation tends to 0.

\section{Analysis of the state equation and the cost functional}
\label{S2}\setcounter{equation}{0}

We recall that a function $u \in L^1(\omega)$ is a function of bounded variation if its distributional derivatives $\partial_{x_i}u$, $1 \le i \le n$, belong to the Banach space of real and regular Borel measures $\mo$. Given a measure $\mu \in \mo$, its norm is given by
\[
\|\mu\|_\mo = \sup\{\int_\omega z\, d\mu : z \in \C \mbox{ and } \|z\|_\C \le 1\} = |\mu|(\omega),
\]
where $\C$ denotes the Banach space of continuous functions $z:\bar\omega \longrightarrow \mathbb{R}$ such that $z  = 0$ on $\partial\omega$, and $|\mu|$ is the total variation measure associated with $\mu$. On the product space $\mon$ we define the norm
\[
\|\mu\|_\mon = \sup\{\int_\omega z\, d\mu : z \in \Cn \mbox{ and } |z(x)| \le 1\ \ \forall x \in \omega\},
\]
where $|\cdot |$ is a norm in $\mathbb{R}^n$.

On $\BV$ we consider the usual norm
\[
\|u\|_\BV = \|u\|_{L^1(\omega)} + \|\nabla u\|_\mon,
\]
that makes $\BV$ a Banach space; see \cite[Chapter~3]{AFP2000} or \cite[Chapter~1]{Giusti84} for details. We recall that the total variation of $u$ is given by
\[
\|\nabla u\|_\mon = \sup\{\int_\omega \div z\, u \, dx : z \in C_0^\infty(\omega)^n \mbox{ and } |z(x)| \le 1\ \ \forall x \in \omega\}.
\]
We also use the notation
\[
\int_\omega |\nabla u| = \|\nabla u\|_\mon,
\]
as already employed in \Pb. For these topologies $\nabla:\BV \longrightarrow \mon$ is a linear continuous mapping.

In the sequel we will denote
\[
a_u = \frac{1}{|\omega|}\int_\omega u(x)\, dx \ \mbox{ and }\ \hat{u}  = u - a_u \ \mbox{ for every } u \in \BV.
\]
By using \cite[Theorem~3.44]{AFP2000} it is easy to deduce that there exists a constant $C_\omega$ such that
\begin{equation}
\|u\| := |a_u| + \|\nabla u\|_\mon \le \max\big(1,\frac{1}{|\omega|}\big)\|u\|_\BV \le C_\omega\|u\|.
\label{E2.1}
\end{equation}
In addition, we mention that $\BV$ is the dual space of a separable Banach space. Therefore every bounded sequence $\{u_k\}_{k = 1}^\infty$ in $\BV$ has a subsequence converging weakly$^*$ to some $u \in \BV$. The weak$^*$ convergence $u_k \stackrel{*}{\rightharpoonup} u$ implies that $u_k \to u$ strongly in $L^1(\omega)$ and $\nabla u_k \stackrel{*}{\rightharpoonup} \nabla u$ in $\mon$; see \cite[pages~124-125]{AFP2000}. We will also use that $\BV$ is continuously embedded in $L^p(\omega)$ with $1 \le p \le \frac{n}{n-1}$, and compactly embedded in $L^p(0,T)$ for every $p < \frac{n}{n - 1}$; see \cite[Corollary~3.49]{AFP2000}. From this property we deduce that the convergence  $u_k \stackrel{*}{\rightharpoonup} u$ in $\BV$ implies that $u_k \to u$ strongly in every $L^p(0,T)$ for all $p < \frac{n}{n - 1}$.

We make the following assumption on the nonlinear term of the state equation. We assume that $f:\Omega \times \mathbb{R} \longrightarrow \mathbb{R}$ is a Borel function, of class $C^2$ with respect to the last variable, and satisfies for almost all $x \in \Omega$
\begin{align}
&f(\cdot,0) \in L^{\hat p}(\Omega)\ \text{ with } \hat p > \frac{n}{2}, \label{E2.2}\\
&\frac{\partial f}{\partial y}(x,y) \ge 0 \quad \forall y \in \mathbb{R},\label{E2.3}\\
&\forall M > 0\, \exists C_M : \left|\frac{\partial f}{\partial y}(x,y)\right| + \left|\frac{\partial^2f}{\partial y^2}(x,y)\right|  \le C_M \quad \forall |y| \le M,\label{E2.4}\\
&\left\{\begin{array}{l}\displaystyle\forall M > 0 \mbox{ and }\forall \rho > 0\, \exists \varepsilon > 0 \mbox{ such that }\\
\displaystyle\left|\frac{\partial^2f}{\partial y^2}(x,y_2) - \frac{\partial^2f}{\partial y^2}(x,y_1) \right| \le \rho \mbox{ if } |y_2 - y_1| < \varepsilon \mbox{ and } |y_1|, |y_2| \le M.\end{array}\right.
\label{E2.5}
\end{align}

Let us observe that if $f$ is an affine function, $f(x,y) = c_0(x)y + d_0(x)$, then \eqref{E2.2}-\eqref{E2.5} hold if $c_0 \ge 0$ in $\Omega$, $c_0 \in L^\infty(\Omega)$, and $d_0 \in L^{\hat{p}}(\Omega)$.

By using these assumptions, the following theorem can be proved in a standard way; see, for instance, \cite[\S 4.2.4]{Troltzsch2010}. For the H\"{o}lder continuity result, the reader is referred to \cite[Theorem 8.29]{Gilbarg-Trudinger83}.
\begin{proposition}
For every $u \in L^{\hat p}(\omega)$ the state equation \eqref{E1.1} has a unique solution $y_u \in C^\sigma(\bar\Omega) \cap H_0^1(\Omega)$ for some $\sigma \in (0,1)$. In addition, for every $M > 0$ there exists a constant $K_M$ such that
\begin{equation}
\|y_u\|_{C^\sigma(\bar\Omega)} + \|y_u\|_{H_0^1(\Omega)} \le K_M\ \ \forall u \in L^{\hat p}(\omega) : \|u\|_{L^{\hat p}(\omega)} \le M.
\label{E2.6}
\end{equation}
\label{P2.1}
\end{proposition}

In the sequel we will denote $Y = C(\bar\Omega) \cap H_0^1(\Omega)$ and $S:L^{\hat p}(\omega) \longrightarrow Y$ the mapping associating to each control $u$ the corresponding state $S(u) = y_u$. We have the following differentiability property of $S$.
\begin{proposition}\label{P2.2}
The mapping $S:L^{\hat p}(\omega) \longrightarrow Y$ is of class $C^2$. For all elements $u, v$ and $w$ of $L^{\hat p}(\omega)$, the functions $z_v = S'(u)v$ and $z_{vw} = S''(u)(v,w)$ are the solutions of the problems
\begin{equation}\label{E2.7}
\left\{\begin{aligned} \displaystyle - \Delta z + \frac{\partial f}{\partial y}(x,y_u)z & = v\chi_\omega & & \mbox{in } \Omega,\\ z & = 0 & & \mbox{on } \Gamma,\end{aligned}\right.
\end{equation}
and
\begin{equation}\label{E2.8}
\left\{\begin{aligned} \displaystyle - \Delta z + \frac{\partial f}{\partial y}(x,y_u)z + \frac{\partial^2f}{\partial y^2}(x,y_u)z_{v}z_{w}& = 0 & & \mbox{in } \Omega,\\ z & = 0 & & \mbox{on } \Gamma,\end{aligned}\right.
\end{equation}
respectively.
\end{proposition}

The proof is a consequence of the implicit function theorem. Let us give a sketch. We define the space
\[
V = \{y \in Y : \Delta y \in L^{\hat p}(\Omega)\}
\]
endowed with the norm
\[
\|y\|_V = \|y\|_{C(\bar\Omega))} + \|y\|_{H_0^1(\Omega)} + \|\Delta y\|_{L^{\hat p}(\Omega)}.
\]
Thus, $V$ is a Banach space. Now we introduce the mapping $\mathcal{F}:V \times L^{\hat p}(\Omega) \longrightarrow L^{\hat p}(\Omega)$ by
\[
\mathcal{F}(y,u) = -\Delta y + f(x,y) - u.
\]
From \eqref{E2.4} we deduce that $\mathcal{F}$ is of class $C^2$ and
\[
\frac{\partial \mathcal{F}}{\partial y}(y,u) z = - \Delta z + \frac{\partial f}{\partial y}(x,y)z.
\]
From the monotonicity condition \eqref{E2.3}, we obtain that $\frac{\partial \mathcal{F}}{\partial y}(y,u) :V \longrightarrow L^{\hat p}(\Omega)$ is an isomorphism. Hence, the implicit function theorem and Proposition \ref{P2.1} with $\omega = \Omega$ imply the existence of a $C^2$ mapping $\hat S:L^{\hat p}(\Omega) \longrightarrow Y$ associating to every element $u$ its corresponding state $\hat S(u) = y_u$. When $\omega \varsubsetneq \Omega$, we use that $S = \hat S \circ S_\omega$, where $S_\omega:L^{\hat p}(\omega) \longrightarrow L^{\hat p}(\Omega)$ is defined by $S_\omega u = u\chi_\omega$. Hence the chain rule leads to the result.

Next, we separate the smooth and the non smooth parts in $J$: $J(u) = F(u) + \alpha G(u) $ with
\[
F(u) = \frac{1}{2}\int_\Omega|y_u - y_d|^2\, dx + \frac{\beta}{2}\big(\int_\omega u(x)\, dx\big)^2 + \frac{\gamma}{2}\int_\omega u^2(x)\, dx \ \mbox{ and } \ G(u) = g(\nabla u),
\]
where $g:\mon \longrightarrow \mathbb{R}$ is given by $g(\mu) = \|\mu\|_\mon$. In the rest of this section we study the differentiability of $F$. From Proposition \ref{P2.2} and the chain rule the following proposition can be obtained.
\begin{proposition}
The functional $F:L^2(\omega) \longrightarrow \mathbb{R}$ is of class $C^2$. The derivatives of $F$ are given by
\begin{equation}
F'(u)v = \int_\omega\Big[\varphi_u(x) + \gamma u(x) + \beta\big(\int_\omega u(s)\, ds\big)\Big]v(x)\, dx,
\label{E2.9}
\end{equation}
and
\begin{equation}
F''(u)(v,w) = \int_\Omega\big(1 - \varphi_u\frac{\partial^2f}{\partial y^2}(x,y_u)\big)z_vz_w\, dx + \gamma\int_\omega vw\, dx + \beta\Big(\int_\omega v\, dx\Big)\Big(\int_\omega w\, dx\Big)
\label{E2.10}
\end{equation}
with $z_v = S'(u)v$, $z_w = S'(u)w$, and $\varphi_u \in Y$ the adjoint state which satisfies
\begin{equation}
\left\{\begin{aligned} \displaystyle - \Delta \varphi_u +  \frac{\partial f}{\partial y}(x,y_u)\varphi_u & = y_u - y_d & & \mbox{in }
\Omega,\\ \varphi_u & = 0 & & \mbox{on }
\Gamma.\end{aligned}\right.\label{E2.11}
\end{equation}
\label{P2.3}
\end{proposition}

The $C(\bar\Omega)$ regularity of $\varphi_u$ follows from the assumptions on $y_d \in L^2(\Omega)$ and the fact that $y_u \in L^\infty(\Omega)$.

\begin{remark}
If $n = 2$, since $\BV$ is embedded in $L^2(\omega)$, then the functional $F:\BV \longrightarrow \mathbb{R}$ is well defined and it is of class $C^2$ with derivatives given by \eqref{E2.9} and \eqref{E2.10}. However, if $n = 3$, then $\BV$ is only embedded in $L^{3/2}(\omega)$. Hence, for elements $u \in \BV$ Proposition \ref{P2.1} is not applicable and, therefore, the functional $F$ is not defined in $\BV$. To deal with the case $n = 3$ we introduced the assumption \eqref{E1.2}, i.e. $\gamma > 0$. Hence, the functional $F:\BV \cap L^2(\omega) \longrightarrow \mathbb{R}$ is well defined and of class $C^2$.

The assumption \eqref{E1.2} can be avoided if we suppose that the nonlinearity $f(x,y)$ has only polynomial growth of arbitrary order in $y$. In this case, Propositions \ref{P2.1} and \ref{P2.2} hold if we change $Y$ to $Y_q = L^q(\Omega) \cap H_0^1(\Omega)$ with $q < \infty$ arbitrarily big. We recall that for a right hand side of the state equation belonging to $L^{3/2}(\Omega)$ the solution of the state equation does not belong to $L^\infty(\Omega)$, in general, even for linear equations. However, since $L^{3/2}(\Omega) \subset W^{-1,3}(\Omega)$, we can use \cite[Theorem 4.2]{Stampacchia65} to deduce that $y_u \in L^q(\Omega)$ $\forall q < \infty$. To analyze the semilinear case one can follow the classical approach of truncation of the nonlinear term, Schauder's fix point theorem, and $L^q$-estimates from the linear case combined with the monotonicity of the nonlinear term. Finally, since $\gamma = 0$, we have that the functional $F:\BV \longrightarrow \mathbb{R}$ is of class $C^2$.
\label{R2.1}
\end{remark}

\begin{remark}
In the state equation, the Laplace operator $-\Delta$ can be replaced by a more general linear elliptic operator with bounded coefficients. All the results proved in this paper hold for these general operators.
\label{R2.2}
\end{remark}

\section{Analysis of the functional $G$}
\label{S3}\setcounter{equation}{0}

Now, we analyze the functional $G$. We already expressed $G$ as the composition $G = g \circ \nabla$. Concerning the functional $g$, we note that it is Lipschitz continuous and convex. Hence, it has a subdifferential and a directional derivative, which are denoted by $\partial g(\mu)$ and $g'(\mu;\nu)$, respectively. Before giving an expression for $g'(\mu;\nu)$, we have to specify the norm that we use in $\mathbb{R}^n$. Indeed, in the definition of the norm $\|\mu\|_\mon$ we have considered a generic norm $| \cdot |$ in $\mathbb{R}^n$. The choice of the specific norm strongly influences the structure of the optimal controls. In this paper, we focus on the Euclidean and the $|\cdot|_\infty$ norms, which lead to different properties for $g$, that we consider separately in the following two subsections. To illustrate one aspect, let us observe that the use of the $|\cdot|_\infty$ norm on $\mathbb{R}^n$ in the definition of $\| \cdot \|_\mon$ implies that
\begin{equation}
\|\mu\|_\mon = \sum_{j = 1}^n\|\mu_j\|_\mo\ \  \forall \mu \in \mon.
\label{E3.1}
\end{equation}
In particular, it holds that
\[
\int_\omega|\nabla u| = \sum_{j = 1}^n\|\partial_{x_j}u\|_\mo\ \ \forall u \in \BV.
\]
However, for the Euclidean norm we have, in general, that
\begin{equation}
\|\mu\|_\mon \neq \Big(\sum_{j = 1}^n\|\mu_j\|^2_\mo\Big)^{1/2}.
\label{E3.2}
\end{equation}
Indeed, the identity \eqref{E3.1} is an immediate consequence of the definitions of the norms $\|\cdot\|_\mo$ and $\|\cdot\|_\mon$. To verify \eqref{E3.2} we give an example. Let us fix $n$ different points $\{\xi^i\}_{i = 1}^n$ in $\omega$ and take $\varepsilon > 0$ small enough such that the balls $B_\varepsilon(\xi^i)$ are disjoint. Now, applying Uryshon's lemma, cf.~\cite[Lemma 2.12]{Rudin70}, we get functions $z_i \in \C$ such that $0 \le z_i(x) \le 1$ $\forall x \in \omega$, $z_i(\xi^i) = 1$ and supp$(z_i) \subset B_\varepsilon(\xi^i)$. We set $z=(z_1,\ldots,z_n)$ and $\mu = (\delta_{\xi^1},\ldots,\delta_{\xi^n})$. Then, since $|z(x)|_2 \le 1$ $\forall x \in \omega$, we have
\[
\|\mu\|_\mon \ge \sum_{i = 1}^n\int_\omega z_i(x)\, d\mu_i(x) = \sum_{i = 1}^n z_i(\xi^i) = n.
\]
On the other hand, we get
\[
\Big(\sum_{j = 1}^n\|\mu_j\|^2_\mo\Big)^{1/2} = \sqrt{n}.
\]

\subsection{The use of the Euclidean norm $|\cdot|_2$}
\label{S3.1}
In order to give an expression for $g'(\mu;\nu)$, let us introduce some notation. We recall that if $\mu \in \mon$, its associated total variation measure is defined as a positive scalar measure as follows
\[
|\mu|(A) = \sup\Big\{\sum_{k = 1}^\infty|\mu(E_k)|_2 : \{E_k\}_k \subset \mathcal{B} \text{ are pairwise disjoint and } A = \bigcup_{k = 1}^\infty E_k\Big\},
\]
where $\mathcal{B}$ is the $\sigma$-algebra of Borel sets in $\omega$, and $|\mu(E_k)|_2$ denotes the Euclidean norm in $\mathbb{R}^n$ of the vector $\mu(E_k)$. Let us denote by $h_\mu$ the Radon-Nikodym derivative of $\mu$ with respect to $|\mu|$. Thus we have
\[
h_\mu \in L^1(\omega,|\mu|),\ |h_\mu(x)|_2 = 1 \text{ for }|\mu|-\text{a.e.}\, x \in \omega \text{ and } \mu(A) = \int_A h_\mu(x)\, d|\mu|(x) \ \forall A \in \mathcal{B}.
\]

Given a second vector measure $\nu \in \mon$, the following Lebesgue decomposition holds: $\nu = \nu_a + \nu_s$, $d\nu_a = h_\nu d|\mu|$, where $\nu_a$ and $\nu_s$ are the absolutely continuous and singular parts of $\nu$ with respect to $|\mu|$, and $h_\nu$ is the Radon-Nikodym derivative of $\nu$ with respect to $|\mu|$. Then, the following identity is fulfilled
\[
\|\nu\|_\mon = \|\nu_a\|_\mon + \|\nu_s\|_\mon = \int_\omega|h_\nu(x)|_2\, d|\mu|(x) + \|\nu_s\|_\mon.
\]
The reader is referred to \cite[Chapter 1]{AFP2000}.

Now, we analyze the subdifferential $\partial g(\mu)$. It is well known that an element $\lambda \in \partial g(\mu)$ if
\begin{equation}
\langle \lambda,\nu - \mu\rangle + \|\mu\|_\mon \le \|\nu\|_\mon\ \ \forall \nu \in \mon.
\label{E3.3}
\end{equation}
This is equivalent to the next two relations
\begin{align}
&\langle\lambda,\mu\rangle = \|\mu\|_\mon, \label{E3.4}\\
&\langle \lambda,\nu\rangle \le \|\nu\|_\mon\ \ \forall \nu \in \mon.\label{E3.5}
\end{align}
Observe that $\lambda$ belongs to the dual of $\mon$, which is not a distributional space. In the special case where $\lambda \in \Cn$, we can establish some precise relations between $\lambda$ and $\mu$. Before proving these relations, let us mention that here we have
\[
\|z\|_\Cn = \sup\{|z(x)|_2 : x \in \omega\}\quad \forall z \in \Cn.
\]

\begin{proposition}
If $\lambda \in \Cn \cap \partial g(\mu)$, then $\|\lambda\|_\Cn \le 1$. Moreover, if $\mu \neq 0$, then the following properties hold
\begin{enumerate}
\item $\|\lambda\|_\Cn = 1$, and
\item {\rm supp}$(\mu) \subset \{x \in \omega : |\lambda(x)|_2 = 1\}.$
\end{enumerate}
\label{P3.1}
\end{proposition}

\begin{proof}
The inequality $\|\lambda\|_\Cn \le 1$ follows from \eqref{E3.5}. Additionally, if $\mu \neq 0$, then \eqref{E3.4} implies 1. To prove 2. we use \eqref{E3.4} as follows
\[
\int_\omega d|\mu|(x) = \|\mu\|_\mon = \langle\lambda,\mu\rangle = \int_\omega\lambda(x)\, d\mu(x) = \int_\omega \lambda(x)\cdot h_\mu(x)\, d|\mu|(x).
\]
Then, using that $|\lambda(x)|_2 \le 1$ $\forall x \in \omega$ and $|h_\mu(x)|_2 = 1$ $|\mu|$-a.e. in $\omega$ we deduce from the identity
\[
\int_\omega d|\mu|(x)  = \int_\omega \lambda(x)\cdot h_\mu(x)\, d|\mu|(x)
\]
that $\lambda(x) \cdot h_\mu(x) = 1$ $|\mu|$-a.e. in $\omega$. Using again that $|h_\mu(x)|_2 = 1$, $|\mu|$-a.e., we conclude that $\lambda(x) = h_\mu(x)$, $|\mu|$-a.e. Therefore, we have that
\[
|\mu|\big(\{x \in \omega : |\lambda(x)|_2 < 1\}\big) = 0,
\]
which implies 2.
\end{proof}

Next we study the directional derivatives of $g$.

\begin{proposition}
Let $\mu, \nu \in \mon$, then
\begin{equation}
g'(\mu;\nu) = \int_\omega h_\nu\, d\mu + \|\nu_s\|_\mon,
\label{E3.6}
\end{equation}
where $\nu = \nu_a + \nu_s = h_\nu d|\mu| + \nu_s$ is the Lebesgue decomposition of $\nu$ respect to $|\mu|$.
\label{P3.2}
\end{proposition}
\begin{proof}
As above, let us write $d\mu = h_\mu d|\mu|$. Then we have
\[
g'(\mu;\nu) = \lim_{\rho \searrow 0}\frac{\|\mu + \rho\nu\|_\mon - \|\mu\|_\mon}{\rho}
 \]
\[
 = \lim_{\rho \searrow 0}\frac{\|\mu + \rho\nu_a\|_\mon + \|\rho\nu_s\|_\mon - \|\mu\|_\mon}{\rho}
\]
\[
= \lim_{\rho \searrow 0}\frac{1}{\rho}\left(\int_\omega|h_\mu(x) + \rho h_\nu(x)|_2\, d|\mu|(x) - \int_\omega|h_\mu(x)|_2\, d|\mu|(x)\right) + \|\nu_s\|_\mon
\]
\[
= \int_\omega\lim_{\rho \searrow 0}\frac{|h_\mu(x) + \rho h_\nu(x)|_2 - |h_\mu(x)|_2}{\rho}\, d|\mu|(x) + \|\nu_s\|_\mon
\]
\[
= \int_\omega \frac{h_\mu(x)\cdot h_\nu(x)}{|h_\mu(x)|_2}\, d|\mu|(x) + \|\nu_s\|_\mon = \int_\omega h_\nu\, d\mu + \|\nu_s\|_\mon.
\]
Since the quotients are dominated by $|h_\nu|_2$, we applied Lebesgue's dominated convergence theorem above. Moreover, we use that $|h_\mu(x)|_2 = 1$ $|\mu|$-a.e.~in $\omega$ in the last equality and also to justify the differentiability of the norm $|\cdot |_2$ at every $h_\mu(x)$ with $x$ in the support of $|\mu|$.
\end{proof}

Now, we come back to the mapping $G$. To this end, let us recall that the adjoint operator $\nabla^*$ is defined by
\[
\nabla^*:[\mon]^*\longrightarrow\BV^*,\quad
\langle \nabla^*\lambda,u\rangle_{\BV^*,\BV} = \langle\lambda,\nabla u\rangle_{[\mon]^*,\mon}.
\]
\begin{proposition}
The following identities hold for all $u \in \BV$:
\begin{align}
&\partial G(u) = \partial(g\circ \nabla)(u) = \nabla^*\partial g(\nabla u),\label{E3.7}\\
&G'(u;v) = (g\circ \nabla)'(u;v) = \int_\omega h_v\, d(\nabla u) + \|(\nabla v)_{s}\|_\mon, \label{E3.8}
\end{align}
where $\nabla v = h_vd|\nabla u| + (\nabla v)_{s}$  is the Lebesgue decomposition of $\nabla v$ with respect to $|\nabla u|$.
\label{P3.3}
\end{proposition}

\begin{proof}
Since $\nabla:\BV \longrightarrow \mon$ is a linear and continuous mapping and $g:\mon \longrightarrow \mathbb{R}$ is convex and continuous, we can apply the chain rule \cite[Chapter~I, Proposition~5.7]{Ekeland-Temam74A} to deduce that $\partial(g \circ \nabla)(u) = \nabla^*\partial g(\nabla u)$, which immediately leads to \eqref{E3.7}.

To verify \eqref{E3.8} it is enough to observe that
\[
(g \circ \nabla)'(u;v) = g'(\nabla u;\nabla v)
\]
and to apply \eqref{E3.6}.
\end{proof}

\subsection{The use of the $|\cdot|_\infty$ norm}
\label{S3.2}

The use of  $|\cdot|_\infty$ norm implies that
\[
\|z\|_\Cn = \sup\{|z(x)|_\infty : x \in \omega\}\quad \forall z \in \Cn.
\]

We recall that every scalar real measure $\mu \in \mo$ admits a Jordan decomposition $\mu = \mu^+ - \mu^-$, where $\mu^+$ and $\mu^-$ are positive measures with disjoint supports. Further, if $h_\mu$ is the Radon-Nikodym derivative of $\mu$ with respect to $|\mu|$, then $\mu^+ = h^+ d|\mu|$ and $\mu^- = h^- d|\mu|$, where $h = h^+ - h^-$ is the decomposition of $h$ in positive and negative parts.

\begin{proposition}
If $\lambda \in \Cn \cap \partial g(\mu)$, then $\|\lambda_j\|_\C \le 1$ for all $j = 1,\ldots,n$. Moreover, if $\mu_j \neq 0$, then the following properties hold
\begin{enumerate}
\item $\|\lambda_j\|_\C = 1$, and
\item {\rm supp}$(\mu_j^+) \subset \{x \in \omega : \lambda_j(x) = +1\}$ and {\rm supp}$(\mu_j^-) \subset \{x \in \omega : \lambda_j(x) = -1\}$.
\end{enumerate}
\label{P3.4}
\end{proposition}

\begin{proof}
Inserting \eqref{E3.1} in \eqref{E3.4} and \eqref{E3.5} we get
\begin{align}
&\sum_{i = 1}^n\langle\mu_i,\lambda_i\rangle = \sum_{i = 1}^n\|\mu_i\|_\mo,\label{E3.9}\\
&\sum_{i = 1}^n\langle\nu_i,\lambda_i\rangle \le \sum_{i = 1}^n\|\nu_i\|_\mo\quad \forall \nu \in \mon.\label{E3.10}
\end{align}
Let us fix $1 \le j \le n$ and take in \eqref{E3.10} $\nu_i = 0$ for every $i \neq  j$ and $\nu_j = \pm\delta_x$ with $x \in  \omega$ arbitrary. Then, we obtain
\[
\pm\lambda_j(x) = \langle\nu_j,\lambda_j\rangle \le \|\nu_j\|_\mo = 1.
\]
This proves that $|\lambda_j(x)| \le 1$ $\forall x \in \omega$ for every $j$. Now, we assume that $\mu_j \neq 0$. From \eqref{E3.9}  we infer
\[
\sum_{i = 1}^n\|\mu_i\|_\mo = \sum_{i = 1}^n\langle\mu_i,\lambda_i\rangle \le \sum_{i = 1}^n\|\mu_i\|_\mo\|\lambda_i\|_\C \le \sum_{i = 1}^n\|\mu_i\|_\mo.
\]
This implies that $\|\lambda_i\|_\C = 1$ for every $i$ such that $\mu_i \neq 0$. Hence, 1.~holds. The second part was proved in \cite[Lemma 3.4]{CCK2013}.
\end{proof}

Now, we compute the directional derivatives of $g'(\mu;\nu)$. Then, we have the following expression which is similar but different from the one obtained in Proposition \ref{P3.2}.

\begin{proposition}
Let $\mu, \nu \in \mon$, then
\begin{equation}
g'(\mu;\nu) = \int_\omega h_\nu\, d\mu + \|\nu_s\|_\mon = \sum_{j = 1}^n\Big\{\int_\omega h_{\nu_j}\, d\mu_j + \|(\nu_j)_s\|_\mo\Big\},
\label{E3.11}
\end{equation}
where $\nu_j = (\nu_j)_a + (\nu_{j})_s = h_{\nu_j} d|\mu_j| + (\nu_j)_s$ is the Lebesgue decomposition of $\nu_j$ with respect to $|\mu_j|$ for $1 \le j \le n$.
\label{P3.5}
\end{proposition}
\begin{proof}
For the proof it is enough use \eqref{E3.1} to obtain
 \[
g'(\mu;\nu) = \lim_{\rho \searrow 0}\frac{\|\mu + \rho\nu\|_\mon - \|\mu\|_\mon}{\rho}
= \sum_{i = 1}^n\lim_{\rho \searrow 0}\frac{\|\mu_i + \rho\nu_i\|_\mon - \|\mu_i\|_\mon}{\rho}.
\]
Then, we proceed as in the proof of \cite[Proposition 3.3]{Casas-Kunisch2014}.
\end{proof}

With the same proof we infer that Proposition \ref{P3.3} is also true for the $|\cdot|_\infty$ norm with \eqref{E3.8} being interpreted as follows
\begin{align}
&G'(u;v) = (g\circ \nabla)'(u;v) = \int_\omega h_v\, d(\nabla u) + \|(\nabla v)_{s}\|_\mon \notag\\
&=\sum_{j = 1}^n\Big\{\int_\omega h_{v,j}\, d(\partial_{x_j}u) + \|(\partial_{x_j}v)_s\|_\mo\Big\},
\label{E3.12}
\end{align}
where $\partial_{x_j} v = h_{v,j}|\partial_{x_j}u| + (\partial_{x_j}v)_s$ is the Lebesgue decomposition of $\partial_{x_j}v$ with respect to $|\partial_{x_j}u|$.

\section{Existence of an optimal control and first order optimality conditions}
\label{S4}\setcounter{equation}{0}

The proof of the existence of an optimal control follows the lines of \cite[Theorem 3.1]{CKK2016} with the obvious modifications.

\begin{theorem}
Let us assume that one of the following assumptions hold.
\begin{enumerate}
\item $\beta + \gamma  > 0$.
\item There exist $q \in [1,2)$ and $C > 0$ such that
\end{enumerate}
\[
\frac{\partial f}{\partial y}(x,y) \le C(1 + |y|^q)\ \ \mbox{for a.a. } x \in \Omega \text{ and } \forall y \in \mathbb{R}.
\]
Then, problem \Pb has at least one solution. Moreover, if $f$ is affine with respect to $y$, the solution is unique.
\label{T4.1}
\end{theorem}

Now, we prove the first order optimality conditions satisfied by any local minimum of \Pb.

\begin{theorem}
Let $\bar u$ be a local solution of \Pb. Then, there exists $\bar\lambda \in \partial g(\nabla\bar u)$ such that
\begin{equation}
\alpha\langle\bar\lambda,\nabla v\rangle_{[\mon]^*,\mon} + \int_\omega\Big(\bar\varphi + \gamma\bar u + \beta\int_\omega\bar u\, dz\Big)v\, dx = 0 \ \ \forall v \in \BV\cap L^2(\omega),
\label{E4.1}
\end{equation}
where $\bar\varphi \in H_0^1(\Omega) \cap C(\bar\Omega)$ is the adjoint state corresponding to $\bar u$.
\label{T4.2}
\end{theorem}

\begin{proof}
Let us denote by  $\bar\varphi \in C(\bar\Omega) \cap H_0^1(\Omega)$ the adjoint state corresponding to the local solution $\bar u$. Given $v \in \BV\cap L^2(\omega)$, from the local optimality of $\bar u$ and the convexity of $G$ we deduce for every $0 < \rho < 1$ small enough
\begin{align*}
&0 \le \frac{J(\bar u + \rho v) - J(\bar u)}{\rho} = \frac{F(\bar u + \rho v) - F(\bar u)}{\rho} + \alpha\frac{G(\bar u + \rho v) - G(\bar u)}{\rho}\\
& \le  \frac{F(\bar u + \rho v) - F(\bar u)}{\rho} + \alpha[G(\bar u +v) - G(\bar u)].
\end{align*}
Passing to the limit as $\rho \to 0$ in the above inequality and using \eqref{E2.9} we obtain for every $v \in \BV$
\[
0 \le \int_\omega\Big(\bar\varphi(x) + \gamma\bar u(x) + \beta\int_\omega \bar u\, ds\Big)v(x)\, dx + \alpha[G(\bar u + v) - G(\bar u)].
\]
Replacing $v$ by $u - \bar u$, this inequality can be written
\[
-\frac{1}{\alpha}\int_\omega \Big(\bar\varphi + \gamma\bar u + \beta\int_\omega\bar u\, ds\Big)(u - \bar u)\, dx + G(\bar u) \le G(u) \quad \forall u \in \BV\cap L^2(\omega).
\]
This along with \eqref{E3.7} implies
\[
-\frac{1}{\alpha}\Big(\bar\varphi + \gamma\bar u + \beta\int_\omega\bar u\, ds\Big) \in \partial G(\bar u) = \nabla^*\partial g(\nabla\bar u).
\]
Hence, there exists $\bar\lambda \in \partial g(\nabla\bar u) \subset [\mon]^*$ such that
\[
\langle\bar\lambda,\nabla v\rangle_{[\mon]^*,\mon} = \frac{-1}{\alpha}\int_\omega\Big[\bar\varphi + \beta\int_\omega\bar u\, ds\Big]v\, dx \ \ \forall v \in \BV \cap L^2(\omega),
\]
which implies \eqref{E4.1}.
\end{proof}

Since $\bar\lambda \in \mon$ and $\mon$ is not a distribution space, sometimes it can be more convenient to handle a different optimality system involving distributional spaces, mainly if we think of the numerical analysis. To this end, we present the following equivalent optimality conditions.

\begin{theorem}
Let us assume that $n = 2$. Given $\bar u \in \BV$, let $\bar y$ and $\bar\varphi$ be the associated state and adjoint state. Then, there exists $\bar\lambda \in \partial g(\nabla\bar u)$ satisfying \eqref{E4.1} if and only if there exists $\bar\Phi \in \Cn$ such that
\begin{align}
&\alpha\langle\nabla v,\bar\Phi\rangle_{\mon,\Cn} + \int_\omega\Big[\bar\varphi + \gamma\bar u + \beta\int_\omega\bar u\, ds\Big]v\, dx = 0 \ \ \forall v \in \BV, \label{E4.2}\\
&\langle\nabla v,\bar\Phi\rangle_{\mon,\Cn} \le \|\nabla v\|_\mon\quad \forall v \in \mon, \label{E4.3}\\
&\langle\nabla\bar u,\bar\Phi\rangle_{\mon,\Cn} = \|\nabla\bar u\|_\mon. \label{E4.4}
\end{align}
\label{T4.3}
\end{theorem}

\begin{proof}
Assume that $\bar\lambda \in \partial g(\nabla\bar u)$ satisfies \eqref{E4.1}. We define a linear form $T_0$ in $\mon$  as follows
\[
D(T_0) = \{\nabla v : v \in \BV\}\ \text{ and }\ T_0(\mu) = \langle\bar\lambda,\nabla v\rangle_{[\mon]^*,\mon} \ \text{ if } \mu = \nabla v.
\]
From \eqref{E3.4} and \eqref{E3.5} we have
\begin{align}
&T_0(\nabla\bar u) =\|\nabla\bar u\|_\mon,\label{E4.5}\\
&T_0(\mu) \le \|\mu\|_\mon \quad \forall \mu \in D(T_0).\label{E4.6}
\end{align}
We prove that $T_0$ is weakly$^*$ continuous on its domain. Let $\{\mu_k\}_k \subset D(T_0)$ and $\mu \in D(T_0)$ be such that $\mu_k \stackrel{*}{\rightharpoonup}\mu$ in $\mon$. By definition of $D(T_0)$ there exists elements $\{v_k\}_k \subset \BV$ and $v \in \BV$ such that $\mu_k = \nabla v_k$ and $\mu = \nabla v$. Without loss of generality we assume that the integrals of each $v_k$ and $v$ in $\omega$ are zero. Using \eqref{E2.1}, we know that $\{v_k\}_k$ is bounded in $\BV$. From the continuity of the embedding $\BV \subset L^2(\omega)$ due to $n = 2$ and the convergence $\nabla v_k \stackrel{*}{\rightharpoonup}\nabla v$ in $\mon$, we obtain that $v_k \rightharpoonup v$ in $L^2(\omega)$. Therefore, we get with \eqref{E4.1}
\begin{align}
&\lim_{k \to \infty}T_0(\mu_k) = \lim_{k \to \infty}\langle\bar\lambda,\nabla v_k\rangle_{[\mon]^*,\mon}\notag\\
& = \lim_{k \to \infty}\frac{-1}{\alpha}\int_\omega\Big[\bar\varphi + \gamma\bar u + \beta\int_\omega\bar u\, ds\Big]v_k\, dx\label{E4.7}\\
& = \frac{-1}{\alpha}\int_\omega\Big[\bar\varphi + \gamma\bar u + \beta\int_\omega\bar u\, ds\Big]v\, dx = \langle\bar\lambda,\nabla v\rangle_{[\mon]^*,\mon} = T_0(\mu),\notag
\end{align}
which implies the weak$^*$ continuity of $T_0$. Hence, there exists a weakly$^*$ continuous linear form $T:\mon \longrightarrow \mathbb{R}$ extending $T_0$; \cite[Theorem 3.6]{Rudin73}. In this case, we know that $T$ can be identified with an element $\bar\Phi \in \Cn$, i.e.
\[
T(\mu) = \langle\mu,\bar\Phi\rangle_{\mon,\Cn} = \int_\omega\bar\Phi\, d\mu\quad  \forall \mu \in \mon;
\]
see \cite[Proposition 3.14]{Brezis2011}. The function $\bar\Phi$ fulfills \eqref{E4.2}--\eqref{E4.4}. Indeed, \eqref{E4.2} follows from the definition of $T_0$ and \eqref{E4.1}, and \eqref{E4.3}-\eqref{E4.4} are the same as \eqref{E4.5}-\eqref{E4.6}.

Reciprocally, assume that $\bar\Phi \in \Cn$ satisfies \eqref{E4.2}--\eqref{E4.4}. This time we define the linear operator
\[
D(T_0) = \{\nabla v : v \in \BV\}\ \text{ and }\ T_0(\mu) = \langle\nabla v,\bar\Phi\rangle_{\mon,\Cn} \ \text{ if } \mu = \nabla v.
\]
From \eqref{E4.3} we know that $T_0$ is a continuous operator in $D(T_0) $ for the strong topology of $\mon$, and $\|T_0\|_{[\mon]*} \le 1$. Hence, the Hahn-Banach theorem implies the existence of an operator $\bar\lambda \in [\mon]*$ extending $T_0$ and such that $\|\bar\lambda\|_{[\mon]*} \le 1$. This along with \eqref{E4.3} implies that
\begin{align*}
&\langle\bar\lambda,\nabla\bar u\rangle = \|\nabla\bar u\|_\mon,\\
&\langle\bar\lambda,\nu\rangle \le \|\nu\|_\mon\ \ \forall \nu \in \mon.
\end{align*}
Hence, we have $\bar\lambda \in \partial g(\nabla\bar u)$; see \eqref{E3.3}--\eqref{E3.5}. Finally, \eqref{E4.1} follows from \eqref{E4.2} and the definition of $T_0$. This concludes the proof.
\end{proof}

\begin{remark}
Theorem \ref{T4.3} is still valid in dimension $n = 3$ if we take $\gamma = 0$ and we assume that the nonlinearity of $f(x,y)$ has a polynomial growth of arbitrary order with respect to the variable $y$; see Remark \ref{R2.1}. Indeed, let us observe that the limit \eqref{E4.8} is still valid because $v_k \rightharpoonup v$ in $L^{3/2}(\Omega)$ and $\bar\varphi + \beta\int_\omega\bar u\, ds$  is a continuous function in $\bar\Omega$.
\label{R4.1}
\end{remark}

\begin{remark}
It would be interesting to prove the existence of a function $\bar\Phi \in \Cn \cap \partial g(\nabla\bar u)$ satisfying \eqref{E4.3}--\eqref{E4.5}. Indeed, Theorem \ref{T4.3} does not guarantee that $\|\Phi\|_\Cn \le 1$. In this hypothetic case, we could deduce from  Propositions \ref{P3.1} and \ref{P3.4} the following sparsity structure of $\nabla\bar u$.
\begin{enumerate}
\item For the $|\cdot|_2$ norm, if $\nabla\bar u \neq 0$ we have $\|\bar\Phi\|_\Cn = 1$ and
\[
{\rm supp}(\nabla\bar u) \subset \{x \in \omega : |\bar\Phi(x)|_2 = 1\}.
\]
\item For the $|\cdot|_\infty$ norm, for any $1 \le j \le n$ such that if $\partial_{x_j}\bar u \neq 0$ we have $\|\bar\Phi_j\|_\C = 1$, and
\[
\begin{array}{l}\displaystyle
{\rm supp}([\partial_{x_j} u]^+) \subset \{x \in \omega : \bar\Phi_j(x) = +1\},\vspace{2mm}\\
\displaystyle {\rm supp}([\partial_{x_j}\bar u]^-) \subset \{x \in \omega : \bar\Phi_j(x) = -1\}.
\end{array}
\]
\end{enumerate}
\label{R4.2}
\end{remark}

\section{Second order optimality conditions}
\label{S5}\setcounter{equation}{0}

The goal of this section is to prove necessary and sufficient second order optimality conditions for problem \Pb. In the whole section, $\bar u$ will denote a fixed element of $\BV \cap L^2(\omega)$ satisfying the optimality conditions given in Theorem \ref{T4.2}. As in Section \ref{S3}, we will distinguish the cases where the norms $|\cdot|_2$ and $|\cdot|_\infty$ in $\mathbb{R}^n$ are used in the definition of the measure $\|\nabla u\|_\mon$.

\subsection{The use of the $|\cdot|_\infty$ norm}
\label{S5.1}
As pointed out in \eqref{E3.1}, the use of the $|\cdot|_\infty$ norm in $\mathbb{R}^n$ leads to the identity
\begin{equation}
\|\nabla v\|_\mon = \sum_{j = 1}^n\|\partial_{x_j} v\|_\mo = \sum_{j = 1}^n\Big\{\int_\omega|h_{v,j}|\, d|\partial_{x_j}\bar u| + \|(\partial_{x_j} v)_s\|_\mo\Big\}
\label{E5.1}
\end{equation}
$\forall v \in \BV$, where $\partial_{x_j}v = h_{v,j}|\partial_{x_j}\bar u| + (\partial_{x_j}v)_s$ is the Lebesgue decomposition of $\partial_{x_j}v$ with respect to the measure $|\partial_{x_j}\bar u|$, $1 \le j \le n$. Moreover, for every $1 \le j \le n$ there exists a Borel function $\bar h_j$ such that
\begin{equation}
|\bar h_j(x)| = 1,\ \ |\partial_{x_j}\bar u|\!-\!\text{a.e.}, \text{ and } \partial_{x_j}\bar u = \bar h_j|\partial_{x_j}\bar u|.
\label{E5.2}
\end{equation}
In the sequel, we will denote $h_v = (h_{v,1},\ldots,h_{v,n})$ and $\bar h = (\bar h_1,\ldots,\bar h_n)$.

First, we state the second order necessary optimality conditions.  To this end we define the cone of critical directions $C_{\bar u}$ as the closure in $L^2(\omega)$ of the cone
\begin{align}
E_{\bar u} &= \{v \in \BV\cap L^2(\omega) : F'(\bar u)v + \alpha G'(\bar u;v) = 0\notag\\
&\qquad\text{and } h_{v,j} \in L^2(|\partial_{x_j}\bar u|), \ 1 \le j \le n\}.
\label{E5.3}
\end{align}
Then, we have the following result.

\begin{theorem}
If $\bar u$ is a local minimum of \Pb, then $F''(\bar u)v^2 \ge 0$ $\forall v \in C_{\bar u}$.
\label{T5.1}
\end{theorem}

\begin{proof}
We will prove the result for every $v \in E_{\bar u}$. Then, the theorem follows by using the continuity of quadratic from $v \in L^2(\omega) \to F''(\bar u)v^2 \in \mathbb{R}$. Given $v \in E_{\bar u}$ and $\rho > 0$ we set
\[
\omega_{\rho,j} = \{x \in \omega : \rho|h_{v,j}(x)| \le \frac{1}{2}\}\quad 1 \le j \le n.
\]

We have with Schwarz inequality
\begin{align*}
&|\partial_{x_j}\bar u|(\omega \setminus \omega_{\rho,j}) \le 2\rho\int_{\omega \setminus \omega_{\rho,j}}|h_{v,j}(x)|\, d|\partial_{x_j}\bar u|\\
&\le 2\rho\sqrt{|\partial_{x_j}\bar u|(\omega \setminus \omega_{\rho,j})}\Big(\int_{\omega \setminus \omega_{\rho,j}}|h_{v,j}(x)|^2\, d|\partial_{x_j}\bar u|\Big)^{1/2},
\end{align*}
which implies
\begin{equation}
\sqrt{|\partial_{x_j}\bar u|(\omega \setminus \omega_{\rho,j})} \le 2\rho\Big(\int_{\omega \setminus \omega_{\rho,j}}|h_{v.j}(x)|^2\, d|\partial_{x_j}\bar u|\Big)^{1/2}\ \ 1 \le j \le n.
\label{E5.4}
\end{equation}

Taking into account \eqref{E5.2} we get for $1 \le j \le n$
\[
\frac{|\bar h_j(x) + \rho h_{v,j}(x)| - |\bar h_j(x)|}{\rho} = h_{v,j}(x) \bar h_j(x)\ \ [|\partial_{x_j}\bar u|]\!-\!\text{a.e.} \ x \in \omega_{\rho,j}.
\]

Using this identity and \eqref{E5.1} we get
\begin{align*}
&\frac{G(\bar u + \rho v) - G(\bar u)}{\rho}\\
& = \sum_{j = 1}^n\Big\{\int_{\omega_{\rho,j}}\frac{|\bar h_j + \rho h_{v,j}| - |\bar h_j|}{\rho}d|\partial_{x_j}\bar u| + \|(\partial_{x_j} v)_s\|_\mon\Big\}\\
& + \sum_{j = 1}^n\int_{\omega \setminus \omega_{\rho,j}}\frac{|\bar h_j + \rho h_{v,j}| - |\bar h|}{\rho}d|\partial_{x_j}\bar u|\\
&= \sum_{j = 1}^n\Big\{\int_{\omega_{\rho,j}}(h_{v,j}\bar h_j)\,d|\partial_{x_j}\bar u| + \|(\partial_{x_j} v)_s\|_\mon\Big\}\\
& + \sum_{j = 1}^n\int_{\omega \setminus \omega_{\rho,j}}\frac{|\bar h_j + \rho h_{v,j}| - |\bar h_j|}{\rho}d|\partial_{x_j}\bar u|\\
&= \sum_{j = 1}^n\Big\{\int_\omega h_{v,j}\,d\partial_{x_j}\bar u + \|(\partial_{x_j} v)_s\|_\mon\Big\}\\
&+ \sum_{j = 1}^n\Big\{\int_{\omega \setminus \omega_{\rho,j}}\frac{|\bar h_j + \rho h_{v,j}| - |\bar h_j|}{\rho}d|\partial_{x_j}\bar u| - \int_{\omega \setminus \omega_{\rho,j}}(h_{v,j} \bar h_j)\,d|\partial_{x_j}\bar u|\Big\}.
\end{align*}
Now, from \eqref{E3.12}, \eqref{E5.2},  Schwarz inequality, and \eqref{E5.4} we infer
\begin{align*}
&\frac{G(\bar u + \rho v) - G(\bar u)}{\rho} \le G'(\bar u;v) + 2\sum_{j = 1}^n\int_{\omega \setminus \omega_{\rho,j}}|h_{v,j}|d|\partial_{x_j}\bar u|\\
&\le G'(\bar u;v) + 2\sum_{j = 1}^n\sqrt{|\partial_{x_j}\bar u|(\omega \setminus \omega_{\rho,j})}\Big(\int_{\omega \setminus \omega_{\rho,j}}|h_{v,j}(x)|^2\, d|\partial_{x_j}\bar u|\Big)^{1/2}\\
&\le G'(\bar u;v) + 4\rho\sum_{j = 1}^n\int_{\omega \setminus \omega_{\rho,j}}|h_{v,j}(x)|^2\, d|\partial_{x_j}\bar u|.
\end{align*}

Next we use the local optimality of $\bar u$. By a Taylor expansion of $F$ around $\bar u$ and using that $v \in E_{\bar u}$, we get for $\rho > 0$ small enough
\begin{align*}
&0 \le J(\bar u + \rho v) - J(\bar u) = \rho[F'(\bar u)v + \alpha G'(\bar u;v)]\\
& + \frac{\rho^2}{2}\Big(F''(\bar u + \theta_\rho v)v^2 + 8\alpha\sum_{j = 1}^n\int_{\omega \setminus \omega_{\rho,j}}|h_{v,j}(x)|^2\, d|\partial_{x_j}\bar u|\Big)\\
&= \frac{\rho^2}{2}\Big(F''(\bar u + \theta_\rho v)v^2 + 8\alpha\sum_{j = 1}^n\int_{\omega \setminus \omega_{\rho,j}}|h_{v,j}(x)|^2\, d|\partial_{x_j}\bar u|\Big)
\end{align*}
with $0 \le \theta_\rho \le 1$. Dividing the above expression by $\rho^2/2$, passing to the limit as $\rho \to 0$, and taking into account that $h_{v.j} \in L^2(|\partial_{x_j}\bar u|)$ and $|\partial_{x_j}\bar u|(\omega \setminus \omega_{\rho,j}) \to 0$, we conclude that $F''(\bar u)v^2 \ge 0$.
\end{proof}

For the sufficient second order conditions we introduce the critical cones
\begin{equation}
C_{\bar u}^\tau=\{v \in \BV\cap L^2(\omega) : F'(\bar u)v + \alpha G'(\bar u;v) \le \tau\|z_v\|_{L^2(\Omega)}\},
\label{E5.5}
\end{equation}
where $\tau > 0$ and $z_v = S'(\bar u)v$. The reader is referred to \cite{Casas2012} for some second order conditions based on these cones; see also \cite{Casas-Troltzsch2014A} and \cite{Casas-Troltzsch2015}. Let us observe that \eqref{E4.1} and the inequality $G'(\bar u;v) \ge \langle\bar\lambda,\nabla v\rangle_{[\mon]^*,\mon}$ imply that $\forall v \in \BV\cap L^2(\omega)$
\begin{equation}
F'(\bar u)v + \alpha G'(\bar u;v) \ge F'(\bar u)v + \alpha\langle\bar\lambda,\nabla v\rangle_{[\mon]^*,\mon} = 0.
\label{E5.6}
\end{equation}

\begin{theorem}
Let $\bar u \in \BV \cap L^2(\omega)$ satisfy the first order optimality conditions stated in Theorem \ref{T4.2} and the second order condition
\begin{equation}
\exists \delta > 0 \text{ and } \exists \tau > 0 : F''(\bar u)v^2 \ge \delta\|z_v\|^2_{L^2(\Omega)}\ \ \forall v \in C^\tau_{\bar u}.
\label{E5.7}
\end{equation}
Then, there exist $\kappa > 0$ and $\varepsilon> 0$ such that
\begin{equation}
J(\bar u) + \frac{\kappa}{2}\|y_u - \bar y\|^2_{L^2(\omega)} \le J(u)\quad \forall u \in \BV \cap L^2(\omega) : \|u - \bar u\|_{L^2(\omega)} \le \varepsilon,
\label{E5.8}
\end{equation}
where $y_u = S(u)$ and $\bar y = S(\bar u)$.
\label{T5.2}
\end{theorem}

\begin{proof}
We follow the proof of \cite[Theorem 3.6]{Casas2012} with some changes. First, from \cite[Lemma 2.7]{Casas2012} we deduce the existence of $\varepsilon_0 > 0$ such that
\begin{equation}
|[F''(u) - F''(\bar u)]v^2| \le \frac{\delta}{2}\|z_v\|^2_{L^2(\Omega)} \ \forall v \in L^2(\omega) \text{ and all } \|u - \bar u\|_{L^2(\omega)} \le \varepsilon_0.
\label{E5.9}
\end{equation}
Moreover, from Proposition \ref{P2.2} we deduce the existence of a constant $C_1 > 0$ such that
\begin{equation}
\|z_v\|_{L^2(\Omega)} = \|S'(\bar u)v\|_{L^2(\Omega)} \le C_1\|v\|_{L^2(\omega)}\ \ \forall v \in L^2(\omega).
\label{E5.10}
\end{equation}
Now, from \eqref{E2.6} we have that there exists a constant $K$ such that $\|y_u\|_{C(\bar\Omega)} \le K$ if $\|u - \bar u\|_{L^2(\omega)} \le \varepsilon_0$. From the adjoint state equation \eqref{E2.11} and \eqref{E2.3} we deduce that $\|\varphi_u\|_{C(\bar\Omega)} \le K'$ for every $\|u - \bar u\|_{L^2(\omega)} \le \varepsilon_0$ and some constant $K'$. Finally, using these estimates, \eqref{E2.4} and the expression \eqref{E2.10} we infer the existence of a constant $C_2 > 0$ such that
\begin{equation}
F''(u)v^2 \ge \gamma\|v\|^2_{L^2(\omega)} - C_2\|z_v\|^2_{L^2(\Omega)}\ \text{ for all } \|u - \bar u\|_{L^2(\omega)} \le \varepsilon_0\text{ and } \forall v \in L^2(\omega).
\label{E5.11}
\end{equation}
Now, we set
\[
\varepsilon = \min\Big\{\varepsilon_0,\frac{2\tau}{(\delta + C_2)C_1}\Big\}
\]
with $\tau$ and $\delta$ given in \eqref{E5.7}. Let $u \in \BV\cap L^2(\omega)$ such that $\|u - \bar u\|_{L^2(\omega)} \le \varepsilon$. We distinguish two cases.\vspace{2mm}

\textit{Case I: $u - \bar u \in C_{\bar u}^\tau$.} Making a Taylor expansion of $F$ around $\bar u$, using the convexity of $G$ and \eqref{E5.6}, \eqref{E5.7} and \eqref{E5.9}, we get for some $0 \le \theta \le 1$
\begin{align}
&J(u) - J(\bar u) \ge [F'(\bar u)(u - \bar u) + \alpha G'(\bar u;u - \bar u)] + \frac{1}{2}F''(\bar u+ \theta(u - \bar u))(u - \bar u)^2\notag\\
&\ge  \frac{1}{2}F''(\bar u)(u - \bar u)^2 +  \frac{1}{2}[F''(\bar u+ \theta(u - \bar u)) - F''(\bar u)](u - \bar u)^2\notag\\
& \ge \frac{\delta}{2}\|z_{u - \bar u}\|_{L^2(\Omega)}^2 - \frac{\delta}{4}\|z_{u - \bar u}\|_{L^2(\Omega)}^2 = \frac{\delta}{4}\|z_{u - \bar u}\|_{L^2(\Omega)}^2.
\label{E5.12}
\end{align}

\textit{Case II: $u - \bar u \not\in C_{\bar u}^\tau$.} This implies that
\begin{equation}
F'(\bar u)(u - \bar u) + \alpha G'(\bar u;u - \bar u) > \tau\|z_{u - \bar u}\|_{L^2(\Omega)}.
\label{E5.13}
\end{equation}
Moreover, from \eqref{E5.10} and the definition of $\varepsilon$ we infer
\[
\|z_{u - \bar u}\|_{L^2(\Omega)} \le C_1\|u - \bar u\|_{L^2(\Omega)} \le \frac{2\tau}{\delta + C_2},
\]
and therefore
\begin{equation}
\frac{\delta + C_2}{2\tau}\|z_{u - \bar u}\|_{L^2(\Omega)} \le 1.
\label{E5.14}
\end{equation}
Using again the convexity of $G$, \eqref{E5.11},  \eqref{E5.13} and \eqref{E5.14} we infer
\begin{align}
&J(u) - J(\bar u) \ge [F'(\bar u)(u - \bar u) + \alpha G'(\bar u;u - \bar u)] + \frac{1}{2}F''(\bar u+ \theta(u - \bar u))(u - \bar u)^2\notag\\
&\ge \tau\|z_{u - \bar u}\|_{L^2(\Omega)} - C_2\|z_{u - \bar u}\|_{L^2(\Omega)}^2\notag\\
& \ge \frac{\delta + C_2}{2}\|z_{u - \bar u}\|^2_{L^2(\Omega)} - \frac{C_2}{2}\|z_{u - \bar u}\|_{L^2(\Omega)}^2 = \frac{\delta}{2}\|z_{u - \bar u}\|^2_{L^2(\Omega)}.
\label{E5.15}
\end{align}

From \eqref{E5.12} and \eqref{E5.15} we deduce that \cite[page 2364]{Casas2012}
\[
J(u) - J(\bar u) \ge \frac{\delta}{4}\|z_{u - \bar u}\|_{L^2(\Omega)}^2  \ \ \forall u \in \BV \cap L^2(\omega) : \|u - \bar u\|_{L^2(\omega)} \le \varepsilon.
\]
Finally, choosing $\varepsilon$ still smaller, if necessary, we have that  \cite[page 2364]{Casas2012}
\[
\frac{1}{2}\|y_u - \bar y\|_{L^2(\Omega)} \le \|z_{u - \bar u}\|_{L^2(\Omega)}   \ \ \forall u \in \BV \cap L^2(\omega) : \|u - \bar u\|_{L^2(\omega)} \le \varepsilon.
\]
The last two inequalities imply \eqref{E5.8} with $\kappa = \frac{\delta}{8}$.
\end{proof}

We observe that \eqref{E5.7} is a sufficient second order condition for strict local optimality of $\bar u$ in the $L^2(\omega)$ sense. Moreover, by using \eqref{E5.8} we can prove stability of the optimal states with respect to perturbations in the data of the control problem. However, it does not provide information on the optimal controls. If $\gamma > 0$ we can change \eqref{E5.7} by a stronger assumption leading to a quadratic growth of the controls instead of the states; i.e. $\|y_u - \bar y\|_{L^2(\Omega)}^2$ can be replaced by $\|u - \bar u\|_{L^2(\omega)}^2$ in \eqref{E5.8}. However, if $\gamma = 0$, then this is not possible; see \cite{Casas2012}.

\begin{theorem}
Suppose that $\gamma > 0$ and let $\bar u \in \BV \cap L^2(\omega)$ satisfy the first order optimality conditions stated in Theorem \ref{T4.2} and the second order condition
\begin{equation}
\exists \delta > 0 \text{ and } \exists \tau > 0 : F''(\bar u)v^2 \ge \delta\|v\|^2_{L^2(\omega)}\ \ \forall v \in C^\tau_{\bar u}.
\label{E5.16}
\end{equation}
Then, there exist $\kappa > 0$ and $\varepsilon> 0$ such that
\begin{equation}
J(\bar u) + \frac{\kappa}{2}\|u - \bar u\|^2_{L^2(\omega)} \le J(u)\quad \forall u \in \BV \cap L^2(\omega) : \|u - \bar u\|_{L^2(\omega)} \le \varepsilon.
\label{E5.17}
\end{equation}
\label{T5.3}
\end{theorem}

\begin{proof}
Using again \cite[Lemma 2.7]{Casas2012} along with \eqref{E5.10} we infer the existence of $\varepsilon > 0$ such that
\begin{equation}
|[F''(u) - F''(\bar u)]v^2| \le \frac{\delta}{2}\|v\|^2_{L^2(\Omega)} \ \forall v \in L^2(\omega) \text{ and all } \|u - \bar u\|_{L^2(\omega)} \le \varepsilon.
\label{E5.18}
\end{equation}
Arguing similarly to \eqref{E5.12}, but using \eqref{E5.16} and \eqref{E5.18} we obtain for every $u \in \BV \cap L^2(\omega)$ such that $\|u - \bar u\|_{L^2(\omega)} \le \varepsilon$ and $u - \bar u \in C_{\bar u}^\tau$
\begin{align}
&J(u) - J(\bar u) \ge [F'(\bar u)(u - \bar u) + \alpha G'(\bar u;u - \bar u)] + \frac{1}{2}F''(\bar u+ \theta(u - \bar u))(u - \bar u)^2\notag\\
&\ge  \frac{1}{2}F''(\bar u)(u - \bar u)^2 +  \frac{1}{2}[F''(\bar u+ \theta(u - \bar u)) - F''(\bar u)](u - \bar u)^2\notag\\
& \ge \frac{\delta}{2}\|u - \bar u\|_{L^2(\omega)}^2 - \frac{\delta}{4}\|u - \bar u\|_{L^2(\omega)}^2 = \frac{\delta}{4}\|u - \bar u\|_{L^2(\omega)}^2.
\label{E5.19}
\end{align}
Thus, \eqref{E5.17} holds with $\kappa = \frac{\delta}{2}$ assuming that $u - \bar u \in C^\tau_{\bar u}$. Now, we argue by contradiction, and we assume that there do not exist $\kappa > 0$ and $\varepsilon > 0$ such that \eqref{E5.17} holds for all the elements $u \in \BV \cap L^2(\omega)$ with $\|u - \bar u\|_{L^2(\omega)} \le \varepsilon$. This implies that for every integer $k > 0$, there exists an element $u_k \in \BV \cap L^2(\omega)$ with
\begin{equation}
\|u_k - \bar u\|_{L^2(\omega)} \le \frac{1}{k}\ \text{ and }\ J(\bar u) + \frac{1}{2k}\|u_k - \bar u\|^2_{L^2(\omega)} > J(u_k).
\label{E5.20}
\end{equation}
From \eqref{E5.19} we know that $u_k - \bar u \not\in C_{\bar u}^\tau$, hence with \eqref{E5.14}
\begin{equation}
F'(\bar u)(u_k - \bar u) + \alpha G'(\bar u;u_k - \bar u) > \tau\|z_{u_k - \bar u}\|_{L^2(\Omega)} \ge \frac{\delta + C_2}{2}\|z_{u_k - \bar u}\|^2_{L^2(\Omega)}
\label{E5.21}
\end{equation}
for every $k$ large enough. Using \eqref{E5.11}, \eqref{E5.20} and \eqref{E5.21} we obtain
\begin{align*}
&\frac{1}{2k}\|u_k - \bar u\|^2_{L^2(\omega)} > J(u_k) - J(\bar u)\\
& \ge [F'(\bar u)(u_k - \bar u) + \alpha G'(\bar u;u_k - \bar u)] + \frac{1}{2}F''(\bar u + \theta_k(u_k - \bar u))(u_k - \bar u)^2\\
&\ge \frac{\delta + C_2}{2}\|z_{u - \bar u}\|^2_{L^2(\Omega)} + \frac{\gamma}{2}\|u_k - \bar u\|^2_{L^2(\omega)} - \frac{C_2}{2}\|z_{u_k - \bar u}\|_{L^2(\Omega)}^2 \ge \frac{\gamma}{2}\|u_k - \bar u\|^2_{L^2(\omega)},
\end{align*}
with is a contradiction because $\gamma > 0$.
\end{proof}

\subsection{The use of the $|\cdot|_2$ norm}
\label{S5.2}

Given an element $v \in \BV$, we consider the Lebesgue decomposition of $\nabla v$ with respect to the positive measure $|\nabla\bar u|$: $\nabla v = h_v d|\nabla\bar u| + (\nabla v)_s$. Hence, we have
\begin{equation}
\|\nabla v\|_\mon = \int_\omega|h_v(x)|_2\, d|\nabla\bar u| + \|\nabla v)_s\|_\mo.
\label{E5.22}
\end{equation}
We also set $\nabla\bar u = \bar h |\nabla\bar u|$, where $|\bar h(x)|_2 = 1$ $|\nabla\bar u|$-a.e.~in $\omega$. Then, we have with \eqref{E3.8}
\begin{equation}
G'(\bar u;v) = \int_\omega(\bar h\cdot h_v)\, d|\nabla\bar u| + \|(\nabla v)_s\|_\mon.
\label{E5.23}
\end{equation}
Now, we define the cone of critical directions
\begin{equation}
C_{\bar u} = \{v \in \BV\cap L^2(\omega) : F'(\bar u)v + \alpha G'(\bar u;v) = 0 \text{ and } |h_v|_2 \in L^2(|\nabla\bar u|)\}.
\label{E5.24}
\end{equation}
Then, we have the following second order necessary optimality conditions.

\begin{theorem}
If $\bar u$ is a local minimum of \Pb, then
\begin{equation}
F''(\bar u)v^2 + \alpha\int_\omega\Big(|h_v(x)|_2^2 - (\bar h(x)\cdot h_v(x))^2\Big)\, d|\nabla\bar u| \ge 0 \quad \forall v \in C_{\bar u}.
\label{E5.25}
\end{equation}
\label{T5.4}
\end{theorem}
\begin{proof}
For fixed $v \in C_{\bar u}$ and given $\rho > 0$, we define
\[
\omega_\rho = \{x \in \omega : \rho|h_v(x)|_2 \le \frac{1}{2}\}.
\]
Arguing as in the proof of Theorem \ref{T5.1} we get the following inequality analogous to \eqref{E5.4}
\begin{equation}
\sqrt{|\nabla\bar u|(\omega\setminus\omega_\rho)} \le 2\rho\Big(\int_{\omega\setminus\omega_\rho}|h_v(x)|^2_2\, d|\nabla\bar u|\Big)^{1/2}.
\label{E5.26}
\end{equation}
Using the differentiability of the $|\cdot|_2$-norm $x \in \mathbb{R}^n \to |x|_2$ for every $x \neq 0$, the fact that $|\bar h(x)|_2 = 1$ $|\nabla\bar u|$-a.e., the Schwarz inequality, and \eqref{E5.26} we get for $0 \le \theta_\rho(x) \le 1$
\begin{align*}
&\frac{G(\bar u + \rho v) - G(\bar u)}{\rho}\\
&=\int_{\omega_\rho}\frac{|\bar h + \rho h_v|_2 - |\bar h|_2}{\rho}\, d|\nabla\bar u| + \|(\nabla v)_s\|_\mon + \int_{\omega \setminus \omega_\rho}\frac{|\bar h + \rho h_v|_2 - |\bar h|_2}{\rho}\, d|\nabla\bar u|\\
&=\int_{\omega_\rho}\Big[\bar h\cdot h_v + \frac{\rho}{2}\Big(\frac{|h_v|^2_2}{|\bar h + \theta_\rho\rho h_v|_2} - \frac{(\bar h + \theta_\rho\rho h_v)\cdot h_v}{|\bar h + \theta_\rho\rho h_v|_2^3}\Big)\Big]\, d|\nabla\bar u|\\
& + \|(\nabla v)_s\|_\mon + \int_{\omega \setminus \omega_\rho}\frac{|\bar h + \rho h_v|_2 - |\bar h|_2}{\rho}\, d|\nabla\bar u|\\
&\le \int_\omega(\bar h\cdot h_v)\, d|\nabla\bar u| + \frac{\rho}{2}\int_{\omega_\rho}\Big[\frac{|h_v|^2_2}{|\bar h + \theta_\rho\rho h_v|_2} - \frac{(\bar h + \theta_\rho\rho h_v)\cdot h_v}{|\bar h + \theta_\rho\rho h_v|_2^3}\Big]\, d|\nabla\bar u|\\
&+ \|(\nabla v)_s\|_\mon + 2\int_{\omega\setminus\omega_\rho}|h_v|_2\, d|\nabla\bar u|\\
& \le \int_\omega(\bar h\cdot h_v)\, d|\nabla\bar u| +  \|(\nabla v)_s\|_\mon\\
&+\frac{\rho}{2}\Big\{\int_{\omega_\rho}\Big[\frac{|h_v|^2_2}{|\bar h + \theta_\rho\rho h_v|_2} - \frac{(\bar h + \theta_\rho\rho h_v)\cdot h_v}{|\bar h + \theta_\rho\rho h_v|_2^3}\Big]\, d|\nabla\bar u| + 8\int_{\omega\setminus\omega_\rho}|h_v|_2^2\, d|\nabla\bar u|\Big\}.
\end{align*}
Using this inequality and the local optimality of $\bar u$, we infer with $u_\rho = \bar u + \theta_\rho\rho v$, $0 \le \theta_\rho \le 1$,
\begin{align*}
&0 \le J(\bar u + \rho v) - J(\bar u) = \rho\Big[F'(\bar u)v + \alpha\frac{G(\bar u + \rho v) - G(\bar u)}{\rho}\Big] + \frac{\rho^2}{2}F''(u_\rho)v^2\\
&\le \rho[F'(\bar u)v + \alpha G'(\bar u;v)] + \frac{\rho^2 }{2}\Big\{F''(u_\rho)v^2\\
&  + \alpha\int_{\omega_\rho}\Big[\frac{|h_v|^2_2}{|\bar h + \theta_\rho\rho h_v|_2} - \frac{(\bar h + \theta_\rho\rho h_v)\cdot h_v}{|\bar h + \theta_\rho\rho h_v|_2^3}\Big]\, d|\nabla\bar u| + 8\alpha\int_{\omega\setminus\omega_\rho}|h_v|_2^2\, d|\nabla\bar u|\Big\}.
\end{align*}
Now, taking into account that $v \in C_{\bar u}$ and dividing the above inequality by $\rho^2/2$ we get
\[
0 \le F''(u_\rho)v^2 + \alpha\int_{\omega_\rho}\Big[\frac{|h_v|^2_2}{|\bar h + \theta_\rho\rho h_v|_2} - \frac{(\bar h + \theta_\rho\rho h_v)\cdot h_v}{|\bar h + \theta_\rho\rho h_v|_2^3}\Big]\, d|\nabla\bar u| + 8\alpha\int_{\omega\setminus\omega_\rho}|h_v|_2^2\, d|\nabla\bar u|.
\]
Finally, using that $|\nabla\bar u|(\omega\setminus\omega_\rho) \to 0$ as $\rho \to 0$, $|\bar h(x)|_2 = 1$, and
\[
\frac{1}{2} \le 1 - \rho|h_v(x)|_2 \le |\bar h + \theta_\rho\rho h_v|_2 \le 1 + \rho|h_v(x)|_2 \le \frac{3}{2} \ \  |\nabla\bar u|\text{-a.e.~in } \omega_\rho,
\]
we pass to the limit as $\rho \to 0$ in the above inequality with the aid of the Lebesgue dominated convergence theorem and we obtain \eqref{E5.25}.
\end{proof}

\begin{remark}
The reader can easily check that Theorems \ref{T5.2} and \ref{T5.3} also hold when the $|\cdot|_2$ norm is used. However, to reduce the gap between the necessary and sufficient conditions for optimality, we should prove that the conditions
\[
F''(\bar u)v^2 + \alpha\int_\omega\Big(|h_v(x)|_2^2 - (\bar h(x)\cdot h_v(x))^2\Big)\, d|\nabla\bar u| \ge \delta\|z_v\|^2_{L^2(\Omega)}\ \ \forall v \in C^\tau_{\bar u}
\]
and
\[
F''(\bar u)v^2 + \alpha\int_\omega\Big(|h_v(x)|_2^2 - (\bar h(x)\cdot h_v(x))^2\Big)\, d|\nabla\bar u| \ge \delta\|v\|^2_{L^2(\omega)}\ \ \forall v \in C^\tau_{\bar u}
\]
imply \eqref{E5.8} and \eqref{E5.17}, respectively. This, however, remains as a challenge.
\label{R5.1}
\end{remark}

\section{A regularization of problem \Pb}
\label{S6}\setcounter{equation}{0}

Here we briefly discuss the effect of an $H^1(\omega)$-regularization term on the first order optimality conditions. For $\epsilon >0$ we consider
\[
   \Pbe \ \ \ \min_{u \in H^1(\omega)} J_\epsilon(u) =  J(u) + \frac{\epsilon}{2} \int_\omega |\nabla u(x)|^2\, dx,
\]
subject to \eqref{E1.1}, and denote a solution by $u_\epsilon$. Let us set
\[
J_\epsilon(u) = F_\epsilon(u) + G(u),
\]
\noindent
where $F_\epsilon(u)=F(u)+ \frac{\epsilon}{2} \int_\omega |\nabla u|^2 \,dx$ for $u \in H^1(\omega)$.  We have
\[
F'_\epsilon(u)v= F'(u)v + \epsilon \int_\omega \nabla u \cdot \nabla v \, dx, \;\text{ and } \;  \partial G(u)= \nabla^* \partial g(\nabla u) \text{ for } u \in H^1(\omega),
\]
where now $\nabla : H^1(\omega) \to L^2(\omega)^n$, and $g: L^2(\omega)^n\to \mathbb{R}$ is given by $g(v)=\|v\|_{L^1(\omega)^n}$.
We have the analog of Theorem \ref{T4.2}, i.e. for every local solution $u_\epsilon$ of $\Pbe$  there exists $\lambda_\epsilon\in \partial G(\nabla u_\epsilon)$ such that
\begin{equation}\label{E6.1}
\alpha (\lambda_\epsilon, \nabla v)_{L^2(\omega)^n} + F'_{\epsilon} (u_\epsilon) v =0, \text{ for all } v \in H^1(\omega).
\end{equation}
Let us focus on $\lambda_\epsilon\in \partial g(\nabla u_\epsilon)$ next. It is equivalent to
\begin{equation}\label{E6.2}
(\lambda_\epsilon, \nabla u_\epsilon)= \|\nabla u_\epsilon\|_{L^1(\omega)^n}, \text{ and }  (\lambda_\epsilon, v) \le  \|v\|_{L^1(\omega)^n} \text{ for all } v \in L^1(\omega)^n.
\end{equation}
\noindent\textbf{The use of the Euclidean norm $|\cdot|_2$}: Here \eqref{E6.2} results in
\begin{equation}\label{E6.3}
\sum_{i=1}^n(\lambda_{\epsilon,i}, \partial_{x_i} u_\epsilon)=  \int_\omega (\sum_{i=1}^n  |\partial_{x_i} u_\epsilon|^2)^{\frac{1}{2}}\, dx, \text{ and }  \sum_{i=1}^n(\lambda_{\epsilon,i}, v_i) \le  \int_\omega (\sum_{i=1}^n  |v_i|^2)^{\frac{1}{2}}\, dx,
\end{equation}
for all $v \in L^1(\omega)^n$.
The second expression in \eqref{E6.3} implies that $\|\lambda_\epsilon\|_{L^\infty(\omega,\mathbb{R}^n)} \le 1$. Moreover, if  $\nabla u_\epsilon \neq 0$,
\begin{equation}\label{E6.4}
\|\lambda_\epsilon\|_{L^\infty(\omega,\mathbb{R}^n)} = 1 \text{ and } {\rm supp}\, \nabla u_\epsilon \subset\{x\in \omega:|\lambda_\epsilon(x)|_2 =1\}.
\end{equation}
The first claim follows from the equality in \eqref{E6.3}. This equality can also be expressed as $\int_\omega |\nabla u_\epsilon|_2 \, dx = \int_\omega (\nabla u_\epsilon \cdot \lambda_\epsilon) \, dx $, which, together with $|\lambda(x)|_2\le 1$ implies the second assertion in \eqref{E6.4}.

\noindent\textbf{The use of the  $|\cdot|_\infty$-norm}: In this case \eqref{E6.2} results in
\begin{equation}\label{E6.5}
\sum_{i=1}^n(\lambda_{\epsilon,i}, \partial_{x_i} u_\epsilon)=  \sum_{i=1}^n  \|\partial_{x_i} u_\epsilon\|_{L^1(\omega)} \; \text{ and } \;   \sum_{i=1}^n(\lambda_{\epsilon,i}, v_i) \le \sum_{i=1}^n  \|v_i\|_{L^1(\omega)},
\end{equation}
for all $v \in L^1(\omega)^n$. This implies that $\|\lambda_{\epsilon,j}\|_{L^\infty(\omega)} \le 1$ for all $j=1,\dots,n$ and if $\partial_{x_j} u_\epsilon \neq 0 $
\begin{equation}\label{E6.6}
\|\lambda_{\epsilon,j}\|_{L^\infty(\omega)}=1,\; \text{ and } \; {\rm supp}\, (\partial_{x_j} u_\epsilon)^{\pm} \subset \{x\in \omega: \lambda_{\epsilon,j}= \pm 1  \}.
\end{equation}
In fact, for any $1\le j \le n$, let $\nu_i=0$ for all $i \neq j$ and $\nu_j= \lambda_{\epsilon,j}$ on $S^+_j=\{x: \lambda_{\epsilon,j} >1\}$, and equal to 0 otherwise. Then $\int_{S^+_j} (\lambda_{\epsilon,j}^2 - \lambda_{\epsilon,j})(x) \, dx \le 0$, while the integrand is strictly positive a.e. Hence $meas( S^+_j)=0$. In an analogous form we exclude the case $\lambda_{\epsilon,j} <-1$, and hence $\|\lambda_{\epsilon,j}\|_{L^\infty(\omega)}\le 1$, for all $j$. Using the first expression in \eqref{E6.5} we have
$$
\sum_{i=1}^n \|\partial_{x_i} u_\epsilon\|_{L^1(\omega)} = \sum_{i=1}^n (\lambda_{\epsilon,i}, \partial_{x_i}u_\epsilon) \le \sum_{i=1}^n \|\partial_{x_i} u_\epsilon\|_{L^1(\omega)},
$$
which implies \eqref{E6.6}.

\noindent\textbf{Asymptotic behavior:} Finally we consider the asymptotic behavior of \eqref{E6.1}, \eqref{E6.2} as $\epsilon \to 0^+$. From the inequality $J_\varepsilon(u_\varepsilon) \le J(0)$ for all $\varepsilon > 0$, we deduce with \eqref{E1.2} the boundedness of  $\{u_\varepsilon\}_\varepsilon$ in $BV(\omega)\cap L^2(\omega)$. Moreover, \eqref{E6.4} and \eqref{E6.6} imply the boundedness of $\{\lambda_\varepsilon\}_\varepsilon$ in $L^\infty(\omega)^n$. Hence there exists $(\bar u, \bar \lambda)\in (BV(\omega)\cap L^2(\omega))\times L^{\infty}(\omega)^n $ such that on a subsequence $(u_\epsilon,  \lambda_\epsilon) \stackrel{*}{\rightharpoonup} (\bar u, \bar \lambda)$ weakly$^*$ in $(BV(\omega)\cap L^2(\omega)) \times L^{\infty}(\omega)$.
Moreover $y_{u_\epsilon} \to y_{\bar u}$ in $L^2(\Omega)$.

Now, given an arbitrary element $u \in H^1(\omega)$, the optimality of $u_\varepsilon$ and the structure of $J$ implies
\[
J(\bar u) \le \liminf_{\varepsilon \to 0}J(u_\varepsilon) \le \limsup_{\varepsilon \to 0}J(u_\varepsilon)  \le \limsup_{\varepsilon \to 0}J_\varepsilon(u_\varepsilon) \le \limsup_{\varepsilon \to 0}J_\varepsilon(u) = J(u).
\]
Since $H^1(\Omega)$ is dense in $\BV \cap L^2(\omega)$, the above inequality implies that $\bar u$ is a solution of \Pb and
\begin{equation}
J(\bar u) = \lim_{\varepsilon \to 0}J(u_\varepsilon) = \limsup_{\varepsilon \to 0}J_\varepsilon(u_\varepsilon) = \inf\Pb = J(\bar u).
\label{E6.7}
\end{equation}
This implies that $J(u_\varepsilon) \to J(\bar u)$ and $\frac{\varepsilon}{2}\int_\omega|\nabla u_\varepsilon|^2\, dx \to 0$. Moreover, from the convergence properties of $\{u_\varepsilon\}_\varepsilon$ and  $\{y_\varepsilon\}_\varepsilon$ we deduce that
\begin{align}
&\hspace{-0.2cm}\lim_{\varepsilon \to 0}\left[\frac{1}{2}\|y_{u_\varepsilon} - y_d\|^2_{L^2(\Omega)} + \frac{\beta}{2}\left(\int_\omega u_\varepsilon\, dx\right)^2\right] = \frac{1}{2}\|y_{\bar u} - y_d\|^2_{L^2(\Omega)} + \frac{\beta}{2}\left(\int_\omega \bar u\, dx\right)^2,
\label{E6.8}\\
&\hspace{-0.2cm}\int_\omega|\nabla\bar u| \le \liminf_{\varepsilon \to 0}\int_\omega|\nabla\bar u_\varepsilon|.
\label{E6.9}
\end{align}
Combining \eqref{E6.8} with the convergence $J(u_\varepsilon) \to J(\bar u)$ we infer
\begin{equation}
\lim_{\varepsilon \to 0}\left(\frac{\gamma}{2}\|u_\varepsilon\|^2_{L^2(\omega)} + \alpha\int_\omega|\nabla u_\varepsilon|\right) = \frac{\gamma}{2}\|\bar u\|^2_{L^2(\omega)} + \alpha\int_\omega|\nabla\bar u|.
\label{E6.10}
\end{equation}

If $\gamma = 0$ then this identity is reduced to $\int_\omega|\nabla u_\varepsilon| \to \int_\omega|\nabla\bar u|$. Let us prove that this convergence property also holds for $\gamma > 0$. Using \eqref{E6.10}, the convergence $u_\varepsilon \rightharpoonup \bar u$ in $L^2(\omega)$, and \eqref{E6.9} we obtain
\begin{align*}
&\frac{\gamma}{2}\|\bar u\|^2_{L^2(\omega)} \le \liminf_{\varepsilon \to 0}\|u_\varepsilon\|^2_{L^2(\omega)} \le \limsup_{\varepsilon \to 0}\|u_\varepsilon\|^2_{L^2(\omega)}\\
&\le \limsup_{\varepsilon \to 0}\left(\frac{\gamma}{2}\|u_\varepsilon\|^2_{L^2(\omega)} + \alpha\int_\omega|\nabla u_\varepsilon|\right) - \alpha\liminf_{\varepsilon \to 0}\int_\omega|\nabla\bar u_\varepsilon|\\
& \le \left(\frac{\gamma}{2}\|\bar u\|^2_{L^2(\omega)} + \alpha\int_\omega|\nabla\bar u|\right) - \alpha\int_\omega|\nabla\bar u| = \frac{\gamma}{2}\|\bar u\|^2_{L^2(\omega)}.
\end{align*}
Therefore, $\|u_\varepsilon\|_{L^2(\omega)} \to \|\bar u\|_{L^2(\omega)}$ holds. Combining this fact with the weak convergence we conclude that $u_\varepsilon \to \bar u$ strongly in $L^2(\omega)$. Inserting this in \eqref{E6.10} it follows that $\int_\omega|\nabla u_\varepsilon| \to \int_\omega|\nabla\bar u|$.

From \eqref{E6.1} we have that
$$
\alpha (\lambda_\epsilon ,\nabla v) + \int_\omega\Big(\varphi(u_\epsilon) + \gamma  u_\epsilon + \beta \int_\omega u_\epsilon \, dz\Big)v\, dx  - \epsilon \int_\omega u_\epsilon \Delta v \,dx = 0, \ \ \forall v \in C_0^\infty(\omega).
$$
Taking the limit $\epsilon \to 0$ we obtain
$$
\alpha (\bar \lambda ,\nabla v) + \int_\omega\Big(\varphi(\bar u) + \gamma  \bar u + \beta \int_\omega \bar u \, dz\Big)v\, dx  = 0, \ \ \forall v \in C_0^\infty(\omega),
$$
which corresponds to \eqref{E4.1}. Moreover, the above relation implies that $\bar \lambda \in L^2_{\div} (\omega)$, and
\begin{equation}\label{E6.11}
-\alpha \div \bar \lambda + \varphi(\bar u) + \gamma  \bar u + \beta \int_\omega \bar u \, dz =0\ \text{ in } L^2(\omega).
\end{equation}
This relation can also be deduced from \eqref{E4.1}. Thus $ \div \, \bar \lambda$
from  Section 4 coincides with  $\div\, \bar \lambda $ obtained by regularisation and it is uniquely defined by \eqref{E6.11}.

From \eqref{E6.1}, the above identity, and the established convergence $\varepsilon\int_\omega|\nabla u_\varepsilon|^2\, dx \to 0$ we find
\begin{align*}
&\lim_{\varepsilon \to 0}(\lambda_\epsilon, \nabla u_\epsilon) = - \lim_{\varepsilon \to 0}\frac{1}{\alpha} F'_\epsilon(u_\epsilon) u_\epsilon = - \lim_{\varepsilon \to 0}\frac{1}{\alpha}\Big(F'(u_\epsilon) u_\epsilon - \epsilon \int_\omega |\nabla u_\epsilon|^2 \, dx\Big)\\
& = - \frac{1}{\alpha} F'(\bar u) \bar u = -(\div \,  \bar \lambda, \bar u).
\end{align*}
Now, from \eqref{E6.2} and the convergence $\int_\omega|\nabla u_\varepsilon| \to \int_\omega|\nabla\bar u|$ we infer
\[
\lim_{\varepsilon \to 0}(\lambda_\epsilon, \nabla u_\epsilon) = \|\nabla\bar u\|_\mon.
\]
From the last two identities, and using again \eqref{E6.2} along with the convergence $\lambda_\varepsilon \stackrel{*}{\rightharpoonup} \bar\lambda$ in $L^\infty(\omega)$ we obtain
$$
(  -\div \, \bar \lambda,  \bar u)=\|\nabla \bar u \|_\mon, \text{ and }
(\bar \lambda, v) \le  |v|_{L^1(\omega)^n} \text{ for all } v \in L^1(\omega)^n.
$$
This corresponds to $
\langle  \bar\lambda, \nabla \bar u\rangle_{[\mon]^*,\mon} = \| \nabla \bar u\|_\mon,$  and
$\langle \bar\lambda,\nu\rangle \le \|\nu\|_\mon $ for all $\nu \in \mon$, which was obtained in Theorem \ref{T4.2}  with $\bar\lambda \in \partial g(\nabla\bar u)\subset [\mon]^*$.

\section{Conclusions}
An analysis for BV-regularised optimal control problems associated to
semilinear elliptic equations was provided. Existence, first  order
necessary and second order sufficient optimality conditions were
investigated.  Special attention was given to the different cases which
arise due to  the choice of a particular vector norm  in the definition
of the BV-seminorm. If \Pb is additionally regularised by an
$H^1(\omega)$-seminorm, then the set where  the gradient of the optimal
solution  vanishes, can be characterised  conveniently by an adjoint
variable, see \eqref{E6.4} and \eqref{E6.6}.
For the original problem \Pb without $H^1(\omega)$-seminorm
regularisation, such a transparent description of the set where the
measure $|\nabla \bar u|$ vanishes is not available, rather it was
replaced by the properties specified in Theorem \ref{T4.3}.

\bibliographystyle{siam}
\bibliography{Casas-Kunisch}
\end{document}